\documentclass[a4paper]{amsart}
\usepackage{amsmath,amsthm,amssymb}
\usepackage[T1]{fontenc}
\usepackage{url}
\usepackage{amsmath,amsthm,amssymb}
\usepackage{verbatim}
\usepackage{mathrsfs}
\usepackage{graphics}
\usepackage{graphicx}
\usepackage{subfigure}

\newtheorem{theorem}{Theorem}[section]
\newtheorem{proposition}[theorem]{Proposition}
\newtheorem{lemma}[theorem]{Lemma}

\theoremstyle{remark}

\newtheorem{example}{Example}

\newtheorem{remark}{Remark}
\newtheorem*{notation}{Notation}

\newcommand{\vep}{\varepsilon}
\newcommand{\R}{{\mathbb{R}}}

\newcommand{\C}{{\mathbb{C}}}
\newcommand{\Z}{{\mathbb{Z}}}

\newcommand{\T}{{\mathbb{T}}}

\let\oldmarginpar\marginpar
\renewcommand\marginpar[1]{\-\oldmarginpar[\raggedleft\footnotesize #1]%
{\raggedright\footnotesize #1}}

\begin{document}
\title[Light rays in arrays of retro-reflector lenses]
{Directional localization of light rays in a periodic array of retro-reflector lenses}
\author[K. Fr\k{a}czek \and M. Schmoll]{Krzysztof Fr\k{a}czek \and Martin Schmoll}

\address{Faculty of Mathematics and Computer Science\\ Nicolaus
Copernicus University\\ ul. Chopina 12/18\\ 87-100 Toru\'n, Poland}
\email{fraczek@mat.umk.pl}
\address{Faculty of Mathematics\\
O-017 Martin Hall\\
Clemson University\\
Clemson, SC 29634 } \email{schmoll@clemson.edu}
\date{\today}

  \keywords{}
\thanks{Research partially supported by the Narodowe Centrum Nauki Grant
2011/03/B/ST1/00407.} \maketitle
\begin{abstract}
We show that vertical light rays in almost every periodic array of Eaton lenses do not leave certain  strips of bounded width. The light rays are traced by
leaves of a non-orientable foliation on a singular plane. We study the flow defined by the induced foliation on  the orientation cover of the singular
plane. The behavior of that flow and ultimately our claim for the light rays is based on an analysis of the Teichm\"uller flow and the Kontsevich-Zorich
cocycle on the moduli space of two branched, two sheeted torus covers in genus two.
\end{abstract}

\section{Introduction}\label{sec:intro}
Trajectories of light can be controlled by systems of mirrors (like in billiard models) and also by changing the refractive index (RI) of a lens. In this
note we will deal with the so called Eaton lens. This is a  retroreflector lens that reflects rays of light back to their sources. More precisely, the Eaton
lens is a round lens (of radius, say $R>0$) where the RI varies from $1$ to infinity  and it is given by $RI=\sqrt{\frac{2R}{r}-1}$ in polar coordinates.
We assume that the refraction index outside the lens is equal to $1$. The RI is not defined at the center of the lens and goes to infinity when approaching
this singular point. The  direction of the light motion is reversed (cf. \cite{Ha-Ha}) after passing through the lens, see Figure~\ref{Eaton_lens}.
\begin{figure}[h]
\includegraphics[width=0.7\textwidth]{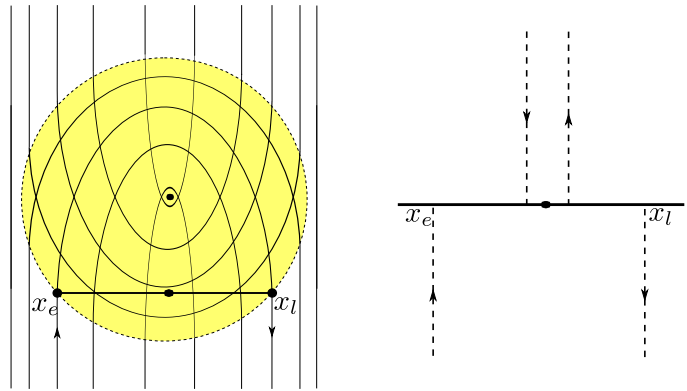}
\caption{Eaton lens and its flat counterpart\label{Eaton_lens}}
\end{figure}
Next let us consider a system of identical Eaton lenses (of
radius $R>0$) that are  arranged on the plane  $\R^2$ so that the
centers  are placed at the points of a lattice
$\Lambda\subset\R^2$. We say that a lattice $\Lambda$ is
$R$-admissible if the circles of radius $R$ centered at the
lattice points of $\Lambda$ are pairwise disjoint.
We will denote such admissible system of Eaton lenses by $L(\Lambda,R)$, see Figure~\ref{lattice_eaton}. Our purpose is to study the behavior of light
orbits for such periodic system of lenses for different pairs of parameters $\Lambda$, $R$. First note that after a rotation we can assume that light rays
running in the same direction are vertical and after rescaling we can assume that $\Lambda$ is unimodular. Denote by $\mathscr{L}$ the space of unimodular
lattices on $\R^2$ that can be identified with the moduli space $SL(2,\R)/SL(2,\Z)$.  Let us consider the natural action of $SL(2,\R)$ given by the left
multiplication and denote by $\mu_{\mathscr{L}}$ the unique probability invariant measure on $\mathscr{L}$. For every $0<R<1/\sqrt{2\sqrt{3}}$ the set of
unimodular $R$-admissible lattices is an open (non-empty, because the hexagonal unimodular lattice is $R$-admissible) subset of $\mathscr{L}$ and hence has
positive measure. For $R\geq 1/\sqrt{2\sqrt{3}}$ the set of unimodular $R$-admissible lattices is empty, compare with the optimal circle packing problem.
The main result of this note is the following.
\begin{theorem}\label{thm:main}
For every $0<R<1/\sqrt{2\sqrt{3}}$  and for a.e.\ $R$-admissible lattice
$\Lambda\in \mathscr{L}$ there exist constants $C=C(\Lambda,R)>0$
and $\theta =\theta(\Lambda,R) \in S^1$, such that every vertical
light ray in $L(\Lambda,R)$ is trapped in an infinite band of
width $C>0$ in direction $\theta$.
%
\end{theorem}
\begin{figure}[h]
\includegraphics[width=0.9\textwidth]{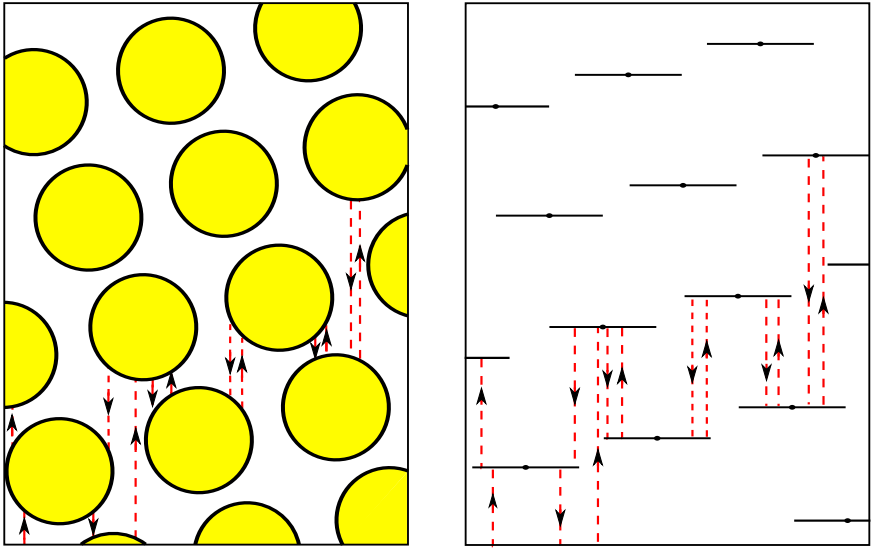}
\caption{$L(\Lambda,R)$ and $F(\Lambda,R)$\label{lattice_eaton}}
\end{figure}

Since we care only about the knowledge about orbits (not on the dynamics) of light rays, we can pass to a simpler model where round lenses are replaced by
flat counterparts, i.e.\ vertical intervals of length $2R$ (called slits),  see Figures~\ref{Eaton_lens}~and~\ref{lattice_eaton}. Such system of flat
horizontal ``lenses'' of length $2R$ whose centers are placed at the points of a lattice $\Lambda$ will be denoted by $F(\Lambda,R)$. The vertical light
rays flow on $F(\Lambda,R)$ by vertical translation with unit speed until hitting the interior of a slit. Then a light ray in $F(\Lambda,R)$ is rotated by
$\pi$ around the center of the slit and runs vertically in the opposite direction until next impact, see the right part of Figure~\ref{Eaton_lens}.

A vertical light ray in $L(\Lambda,R)$ entering to an Eaton lens at a
point $x_e$ leaves the lens at $x_l$ so that the light is rotated by $\pi$ around the center of the interval $[x_e,x_l]$ which is a horizontal chord
of the lens, see the left part of Figure~\ref{Eaton_lens}. There is one exception to this rule when a light ray approaches the center of the lens. Then the light ray does not leave the lens.
We adopt the convention that such a light ray turns back at the center of the lens.
Under this convention for every light orbit in $L(\Lambda,R)$ there is a corresponding orbit in $F(\Lambda,R)$ such that both orbits coincide outside the
lattice of circles. Since inside the circles the distance between these orbits is bounded by $2R$, the distance between the whole corresponding light
orbits in $L(\Lambda,R)$ and $F(\Lambda,R)$ is bounded by $2R$, as well.

In fact, we will deal with systems $F(\Lambda,R)$ for $R>0$ and $\Lambda\in\mathscr{L}$ such that
\begin{equation}\label{cond:adm}
\text{the slits in $F(\Lambda,R)$ are pairwise disjoint.}
\end{equation}
Of course, $R$-admissibility of $\Lambda$ implies this condition.

The problem of understanding the behavior of vertical light rays in $F(\Lambda,R)$ is reduced  in Section~\ref{sec:red} to the study of the vertical flow
on a translation surface $\widetilde{M}(\Lambda,R)$ which is a $\Z^2$-cover of a compact translation surface $M(\Lambda,R)$ the union of two slit tori. The
passage to the framework of translation surfaces allows us to exploit a powerful approach related  to Teichm\"uller dynamics and Lyapunov exponents of the so
called Kontsevich-Zorich cocycle. Exploiting the phenomenon of bounded deviation discovered by Zorich in \cite{Zor1,Zor2} we prove the following.
\begin{theorem}\label{thm:mainflat}
For every $R>0$ and for a.e.\ lattice $\Lambda\in \mathscr{L}$
there exist constants $C=C(\Lambda,R)>0$ and $\theta
=\theta(\Lambda,R) \in S^1$, such that every vertical light ray in
$F(\Lambda,R)$ is trapped in an infinite band of width $C>0$ in
direction $\theta$.
\end{theorem}
If $\Lambda$ is $R$-admissible then every light orbit in
$L(\Lambda,R)$ has a corresponding orbit in $F(\Lambda,R)$ so that
the distance (in $\R^2$) between them is bounded by $2R$.
Therefore,  Theorem~\ref{thm:main} follows directly from
Theorem~\ref{thm:mainflat}.
\medskip

\noindent \emph{Outline of strategy.}
We have already seen that the Eaton lens dynamics converts to a dynamics on a plane with ``slit reflectors'' preserving the dynamical features which matter
for our problem. The group $\Z^2$ acts on the slit plane by translations, so let us look at the simplest plane with a $\Z^2$-action, the complex plane.
Take a curve $\widetilde{\gamma}:[0,1] \rightarrow \C$ on the complex plane and its image $\gamma:=p  \circ \widetilde{\gamma}: [0,1] \rightarrow \T^2$ on
the the quotient $ \C \stackrel{p }{\rightarrow} \T^2\cong \R^2/\Z^2$ modulo $\Z^2$. To study properties of $\widetilde{\gamma}$, it is enough to look at
the curve  $\gamma$ on $\T^2$ and study its lifts.  Suppose we want to know the location of the endpoint $\widetilde{\gamma}(1)$ of a particular lift
$\widetilde{\gamma}$. To do this we tile the complex plane by $\Z^2$ translates of the unit square $[0,1)^2$, the fundamental domain  representing $\T^2$
and fix an enumeration of the tiles by $\Z^2$. Given that $\widetilde{\gamma}(1) \in p ^{-1}(\gamma(1))$, all we need to locate $\widetilde{\gamma}(1)$ is
the integer coordinate of the tile containing it. That information can be derived from $\gamma$ and the topology of $\T^2$. In fact, the coordinate of the tile containing $\gamma(1)$ is obtained from the tile coordinate of $\gamma(0)$ by adding the number of (oriented) crossings of $\gamma$ with the images of horizontal tile edges and vertical tile edges on $\T^2$ with respect to $p$. This can be calculated as an algebraic intersection number of $\gamma$ (after closing it up to a loop without generating new edge intersections) with the respective homology classes defined by vertical and horizontal edges.

Clearly the geometry and the geodesic dynamics on the slit plane, see Figure \ref{lattice_eaton} on the right, differs from the euclidean one on $\C$
and is certainly more complex. Following the previous idea we want to calculate the intersection numbers of an orbit with the vectors generating the
lattice translation symmetry of the slit plane. Those intersection numbers, or more precisely their asymptotic behavior, turn out to be sufficient to show
our claims. Some technical difficulties emerge. One is that foliations on the quotient tori of the slit planes are not orientable, in the sense that they
do not define flows. Generally for this kind of non orientable foliation on a surface, say $S$, there exists a unique double cover $M \rightarrow S$ the
\emph{orientation cover}, such that the pulled back foliation on $M$ is orientable and hence defines a flow. While the choice of two homology classes on a
torus for intersection calculations is more or less canonical, we need to isolate the right homology classes on the orientation cover.

Our general method to study the long term behavior of (vertical) leaves applies to various cases.  To describe it we start with a less general case, in
which the quotient torus, say $T^2$, carries an orientation preserving homeomorphism  $\phi: T^2 \rightarrow T^2$ which is locally affine linear, a so
called \emph{affine homeomorphism}. Any affine homeomorphism, has constant derivative $D\phi \in \text{PSL}_2(\R)$ (except for a discrete set of points).
We further need $\phi$ to be a {\em pseudo-Anosov} map, that is $D\phi$ is hyperbolic, which is the case if $|\text{tr}(D\phi)|>2$. A pseudo-Anosov has two
eigendirections, defining the stable and unstable eigenfoliations on $T^2$. The leaves of the unstable foliation are expanded under the application of $\phi$,
while the leaves of the stable foliation are contracted. Up to a conjugation with a convenient affine linear transformation we may actually assume the
vertical direction on $T^2$ is the stable eigendirection of a pseudo-Anosov on $T^2$.

Consider the homology class of a loop defined by closing up a segment of a vertical leaf.  Then the key step in determining the intersection numbers is to calculate
the induced map $\phi_{\ast}: H_1(T^2,\R) \rightarrow H_1(T^2,\R)$ in homology. We can
extract  some of the information on the shape of
light rays from $\phi_{\ast}$. If $\phi_{\ast}$ is hyperbolic, the vertical direction is  confined in a strip, moreover its stable eigendirection gives (together with some
coordinate adjustment) the direction of the confining strip. This uses the phenomenon of bounded deviation discovered by Zorich, which applies to the
vertical flow on the orientation cover $M_{T^2} \rightarrow T^2$. Note that the existence of pseudo-Anosov on $T^2$ implies the existence of pseudo-Anosovs
on  $M_{T^2}$.
\smallskip

Surfaces which have vertical foliations stabilized by a pseudo-Anosov are rather rare and in order to show the claim of Theorem~\ref{thm:mainflat} we need
to consider a larger set of surfaces.  The connection with the pseudo-Anosov case is made by the observation that orientation preserving homeomorphisms
act on certain sets of flat surfaces with fixed topological data. Recall that we want to study dynamics on surfaces which are branched torus covers of degree two
with two ramification points, see Figures \ref{surfcomcover} and \ref{surfcomgenform}, so let us consider the set of those covers. Up to isomorphism
this set of torus covers defines a certain locus, say $\mathcal{M}$, in a stratum of moduli space. A point in $\mathcal{M}$ represents a surface equipped
with a holomorphic one form determining the geometry and dynamics on the surface. We will move a surface around in $\mathcal{M}$, by applying the one
parameter subgroup $(g_t)_{ t\in \R}$ of $\text{SL}_2(\R)$, where $g_t = \left[\begin{smallmatrix} e^t & 0 \\ 0 & e^{-t} \end{smallmatrix} \right]$. The
orbit of $M \in \mathcal{M}$ under $(g_t)_{ t\in \R}$ is a {\em Teichm\"uller geodesic}.
The locus $\mathcal{M}$ is connected, carries a natural orbifold structure and also admits a natural flow invariant ergodic finite measure equivalent to
Lebesgue measure in local coordinates.

To study the asymptotic behavior of homology classes along Teichm\"uller geodesics one needs to replace the first homology group of a surface by a global
object over $\mathcal{M}$. The {\em (homological) Hodge bundle} $\mathcal{H}$ over $\mathcal{M}$ is the bundle having as fiber over $M \in \mathcal{M}$ the
first homology group of $M$. The object describing how homology classes change along geodesics and this is the Kontsevich-Zorich cocycle $G^{KZ}_t:
\mathcal{H} \rightarrow \mathcal{H}$. Ergodicity of the invariant measure on $\mathcal{M}$ allows us to apply Oseledet's theorem and as a consequence the
Kontsevich-Zorich cocycle has Lyapunov exponents. We only need particular Lyapunov exponents for a flow invariant sub-bundle characterized by the homology
classes defining the infinite cover. Those Lyapunov exponents are known and were calculated by Bainbridge \cite{Bain}. This strategy produces
Theorem~\ref{thm:mainflat}, i.e.\  the existence of a  common trend for vertical light rays in $F(\Lambda,R)$ when the radius $R$ is fixed and the choice
of lattice $\Lambda$ is random.

Our initial pseudo-Anosov example can be seen as a special case of this argument. In fact pseudo-Anosov maps appear as $g_{t_0}$, where $t_0>0$ is a period
of a periodic Teichm\"uller geodesic. The previous strategy applies, if we restrict the respective objects defined over $\mathcal{M}$ to the closed
geodesic and replace the Liouville measure on $\mathcal{M}$ by the flow invariant probability measure supported on the periodic orbit. Applying this
observation, we show in Section~\ref{sec:examples} that the vertical direction on $L(\Lambda,R)$, with $R=1/3$ and
$\Lambda=(1,0)\,\Z+((3+\sqrt{21})/6,1)\,\Z$, is a
 pseudo-Anosov eigendirection and so every vertical
light ray in $L(\Lambda,R)$ is trapped in a band. We also show that every band has slope $-(\sqrt{21}+3\sqrt{5})/4$.

\begin{remark}
The authors believe that a stronger version of Theorem~\ref{thm:main} is true, namely for every $R$-admissible lattice $\Lambda$ and for almost every
direction $\theta\in S^1$ all light rays on $L(\Lambda,R)$ in the direction $\theta$ are trapped in bands. However, we expect that its proof needs a much
more advanced approach than used in the present work.
\end{remark}

\section{From lens lattices to translation surfaces}\label{sec:red}

At the beginning of this section we briefly recall some basic
notions related to translation surfaces and their $\Z^d$-covers.
For further background material  we refer the reader to
\cite{Ful,Ho-We,Ma,ViB}.

\subsection{Translation surfaces and their $\Z^d$-covers}
A translation surface is a pair $(M,\omega)$ where $M$ is an orientable Riemann surface  (not necessarily compact)
and  $\omega$ is a translation structure on $M$, that is a non-zero holomorphic $1$-form also called \emph{Abelian differential}.
Let $\Sigma=\Sigma_\omega\subset M$ denote the set of zeros of $\omega$ which are also the
\emph{singular points} of the translation structure. For every $\theta\in S^1 =
\R/2\pi \Z $ denote by $X_\theta=X^{\omega}_\theta$ the vector field in direction $\theta$ on $M\setminus\Sigma$, i.e.\
$\omega(X^{\omega}_\theta)=e^{i\theta}$. Then the corresponding directional flow $(\varphi^{\theta}_t)_{t\in\R}=(\varphi^{\omega,\theta}_t)_{t\in\R}$, also known as \emph{translation flow}, on $M\setminus\Sigma$ preserves the volume form
$\nu_{\omega}=\frac{i}{2}\omega\wedge\overline{\omega}=\Re(\omega)\wedge\Im(\omega)$. We will use the notation $(\varphi^{v}_t)_{t\in\R}$  for the
\emph{vertical flow} (corresponding to $\theta = \frac{\pi}{2}$). If $M$ is compact let us denote the area of $(M,\omega)$ by  $A(\omega)=\nu_\omega (M)$.

Let $(M,\omega)$ be a compact connected translation surface.  A
{\em $\Z^d$-cover} of $M$ is a surface $\widetilde{M}$ with a free
totally discontinuous action of the group $\Z^d$ such that the
quotient manifold $\widetilde{M}/\Z^d$ is homeomorphic to $M$.
Then the projection $p:\widetilde{M}\to M$  is called a {\em
covering map}. Denote by $\widetilde{\omega}$ the pullback of the
form $\omega$ by the map $p$. Then
$(\widetilde{M},\widetilde{\omega})$ is a translation surface as
well.
\begin{remark}\label{rem:cover}
Up to isomorphism $\Z^d$-covers of $M$  are in one-to-one
correspondence with  $H_1(M,\Z)^d$. For any elements $\xi_1,\xi_2\in
H_1(M,\Z)$ denote by $\langle \xi_1,\xi_2 \rangle $ the algebraic
intersection number of $\xi_1$ with $\xi_2$.

The  $\Z^d$-cover $\widetilde{M}_\gamma$ determined by $\gamma\in H_1(M,\Z)^d$ has
the following properties:

If $\sigma$ is a closed curve in $M$, $[\sigma]\in H_1(M,\Z)$ and
\[\bar{n}=(n_1,\ldots,n_d):=(\langle [\sigma],\gamma_1 \rangle,\ldots,\langle [\sigma],\gamma_d\rangle ) \in
\mathbb{Z}^d\]
then $\sigma$  lifts to a path $\widetilde{\sigma}: [t_0, t_1]\to \widetilde{M}_\gamma$
such that $\widetilde{\sigma}(t_1) = \bar{n} \cdot
\widetilde{\sigma}(t_0)$, where $\cdot$ denotes the action of
$\Z^d$ on $\widetilde{M}_\gamma$.
\end{remark}
\subsection{A translation surface associated to $F(\Lambda, R)$}
As in the case of rational billiards, let us consider a flow
describing the dynamics of vertical light rays in $F(\Lambda,R)$.
Let us label the slits of $F(\Lambda,R)$ by elements of $\Z^2$.
Since the directions of  such orbits are either positive or negative, the
phase space of the flow consists of two copies of $F(\Lambda,R)$,
one $F_+(\Lambda,R)$ for positive and one $F_-(\Lambda,R)$ for
negative orbit segments. Denote by
$\zeta_{\pm}:F_{\pm}(\Lambda,R)\to F(\Lambda,R)$ the map
establishing a natural identification of each copy with
$F(\Lambda,R)$.
\begin{figure}[h]
\includegraphics[width=0.8\textwidth]{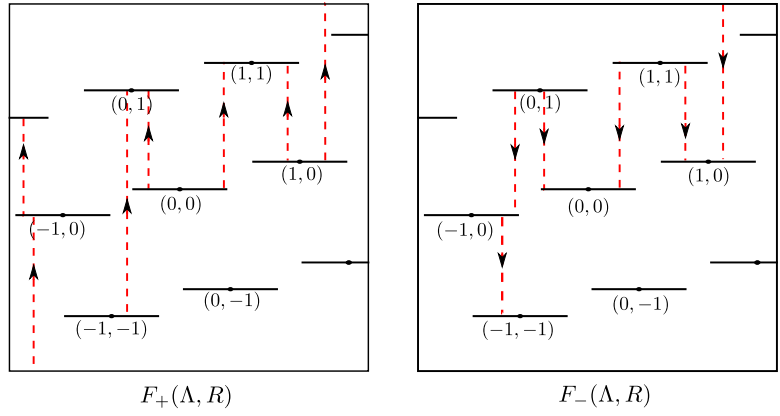}
\caption{The phase space of the light rays flow}\label{lrf1}
\end{figure}
The light ray flow $(\widetilde{\varphi}_t)_{t\in\R}$ acts  on
each point of the phase space moving it vertically (in positive
or negative direction) with unit speed until it hits the
interior of a slit (flat lens), then the point is rotated around
the center of the slit by the angle $\pi$ and it changes
from copy $F_{\pm}(\Lambda,R)$ to copy $F_{\mp}(\Lambda,R)$, see
Figure~\ref{lrf1}.
\begin{figure}[h]
\includegraphics[width=0.8\textwidth]{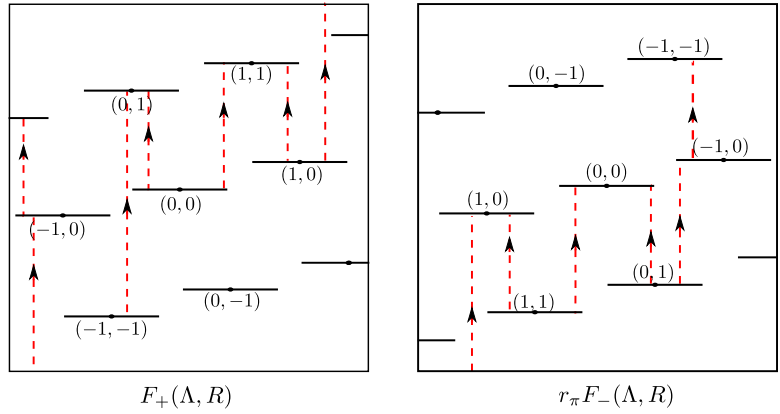}
\caption{Components of the surface $\widetilde{M}(\Lambda, R)$}\label{lrf2}
\end{figure}
Let us rotate the copy $F_-(\Lambda,R)$ by the angle $\pi$ around the center of $(0,0)$-th slit
(denote this rotation by $r_\pi$),  see
Figure~\ref{lrf2}. Next glue the top (bottom) of the $(m,n)$-th
slit in $F_+(\Lambda,R)$ to the bottom (top) of the $(m,n)$-th
slit in $r_{\pi}F_-(\Lambda,R)$ for every $(m,n)\in\Z^2$. The
resulting surface will be denoted by $\widetilde{M}(\Lambda,R)$.
The surface $\widetilde{M}(\Lambda,R)$ carries a natural
translation structure $\widetilde{\omega}$ whose
restrictions to $F_+(\Lambda,R)$ and $r_{\pi}F_-(\Lambda,R)$
are defined by $dz$. Then the zeros of $\widetilde{\omega}$ (all of
order one) arise from the ends of the slits. Moreover, the light
rays flow $(\widetilde{\varphi}_t)_{t\in\R}$ regarded as a flow on
$\widetilde{M}(\Lambda,R)$ is the translation flow in the vertical
direction.
\begin{figure}[h]
\includegraphics[width=0.8\textwidth]{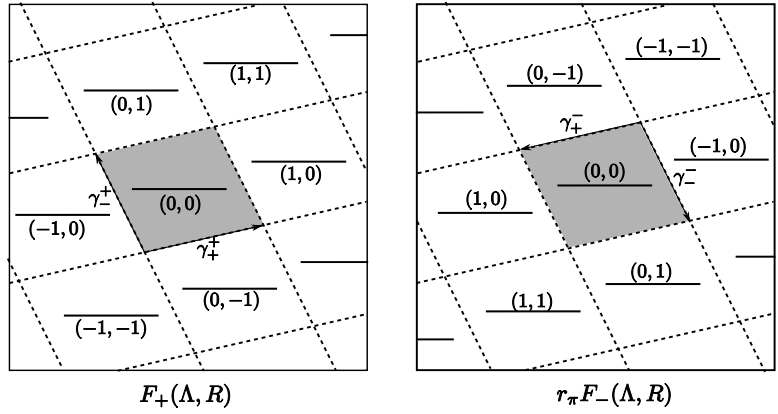}
\caption{A domain for the  $\Lambda$-action on
$\widetilde{M}(\Lambda, R)$}\label{lrf3}
\end{figure}
Let us consider a free totally discontinuous action of $\Lambda$
on $\widetilde{M}(\Lambda,R)$ given  by
\[\lambda\cdot \widetilde{x}=
  \begin{cases}
    \zeta_+^{-1}(\zeta_+(\widetilde{x})+\lambda) & \text{ if }\widetilde{x}\in F_+(\Lambda,R), \\
    r_\pi\circ\zeta_-^{-1}((\zeta_-\circ r_\pi^{-1}(\widetilde{x}))+\lambda) & \text{ if }\widetilde{x}\in r_\pi F_-(\Lambda,R).
  \end{cases}
\]
Since this action preserves the form $\widetilde{\omega}$, we can
consider the quotient translation surface which will be denoted by
$M(\Lambda,R)$.

\subsection{A convenient representation of $M(\Lambda,R)$.}
In this section we describe a representation of the translation surface
$M(\Lambda,R)$ such that its $\Z^2$-cover $\widetilde{M}(\Lambda,R)$
has a convenient form.

Suppose that $R>0$ and $\Lambda\subset \R^2$ is a unimodular lattice satisfying \eqref{cond:adm}. Then there exists a \emph{positive} basis $\gamma_+,
\gamma_-$ of $\Lambda$, i.e.\ a basis with $\gamma_+\in\R_{++}$ and $\gamma_-\in\R_{-+}$, where
\[
\R_{++}=\{(x,y)\in\R^2:x>0,y\geq 0\},\quad
\R_{-+}=\{(x,y)\in\R^2:x\leq 0,y>0\}.
\]
Let us consider the centered parallelogram
\[P(\gamma_+,\gamma_-)=[-1/2,1/2)\gamma_++[-1/2,1/2)\gamma_-\]
generated by vectors $\gamma_+,\gamma_-\in\R^2$.
\begin{lemma}
Let $R>0$ and let $\Lambda$ be a unimodular lattice so  that \eqref{cond:adm} is valid.
Then there exists a positive basis $\gamma_+$, $\gamma_-$ of
$\Lambda$ such that
\begin{equation}\label{eq:contint}
[-R,R]\times\{0\}\text{ is a subset of the interior of }P(\gamma_+,\gamma_-).
\end{equation}
\end{lemma}
\begin{proof}
Since the intersection of $P(\gamma_+,\gamma_-)$ and the line $\R\times\{0\}$ is a symmetric horizontal interval of length $1/\max(\gamma_2^+,\gamma_2^-)$
($\gamma_2^\pm$ is the second coordinate of $\gamma_\pm$), we need find a positive basis $\gamma_+$, $\gamma_-$ such that
$0\leq\gamma_2^+,\gamma_2^-<\frac{1}{2R}$. Such a basis can be found using an Euclidean type algorithm starting from any positive basis $a^0$, $b^0$ of
$\Lambda$. Indeed, let us consider the sequence $(a^n, b^n)_{n\geq 0}$ of positive bases of $\Lambda$ defined inductively by:
\begin{align*}
a^{n+1}=a^n-b^n,\ b^{n+1}=b^n& \text{ if }a^n-b^n\in\R_{++}\\
a^{n+1}=a^n,\ b^{n+1}=b^n-a^n& \text{ if }b^n-a^n\in\R_{-+}.
\end{align*}
If there exists $n\geq 0$ such that $a^n_2=0$ then, by
\eqref{cond:adm}, $a^n_1>2R$. Since $\Lambda$ is unimodular, we
have $1=a^n_1b^n_2-a^n_2b^n_1=a^n_1b^n_2$. Therefore,
$b^n_2<\frac{1}{2R}$ and hence $a^n$, $b^n$ is a required positive
basis.

Now suppose that $a^n_2>0$ for every $n\geq 0$. By definition, the
sequences $(a^n_2)_{n\geq 0}$ and $(b^n_2)_{n\geq 0}$ are
non-increasing and hence $a^n_2\to a\geq 0$ and $b^n_2\to b\geq
0$. Since $(a^n_2)_{n\geq 0}$ and $(b^n_2)_{n\geq 0}$ are both
positive, we have $a_2^{n+1}=a_2^n-b_2^n$ for infinitely many
$n\geq 0$ and $b_2^{n+1}=b_2^n-a_2^n$ for infinitely many $n\geq
0$. It follows that $a=a-b$ and $b=b-a$, so $a=b=0$.  Therefore,
we can find $n\geq 0$ with $a^n_2,b^n_2<\frac{1}{2R}$. Then $a^n$,
$b^n$ is a required positive basis.
\end{proof}

Suppose that $\gamma_+$, $\gamma_-$ is a positive basis of $\Lambda$ satisfying \eqref{eq:contint} and
consider the action of the lattice $\Lambda$ on $F(\Lambda,R)$ by translations. Then $P(\gamma_+,\gamma_-)$ is a fundamental domain for the $\Lambda$-action on $F(\Lambda,R)$ and $P(\gamma_+,\gamma_-)$
contains exactly one slit.

Let $\zeta:\widetilde{M}(\Lambda,R)\to F(\Lambda,R)$ be the map given by
\begin{equation}\label{def:zeta}
\zeta (\widetilde{x})=
  \begin{cases}
    \zeta_+(\widetilde{x}) & \text{ if }\widetilde{x}\in F_+(\Lambda,R), \\
    \zeta_-\circ r_\pi^{-1}(\widetilde{x}) & \text{ if }\widetilde{x}\in r_\pi F_-(\Lambda,R).
  \end{cases}
\end{equation}
Then $\zeta$ is two-to-one and, by the definition of the
$\Lambda$-action on $\widetilde{M}(\Lambda,R)$, we have
\begin{equation}\label{def:act}
\zeta(\lambda\cdot
\widetilde{x})=\zeta(\widetilde{x})+\lambda\quad\text{ for all
}\quad\lambda\in\Lambda\text{ and } \widetilde{x}\in
\widetilde{M}(\Lambda,R).
\end{equation}
Therefore, the set
\begin{equation}\label{def:domainD}
D:=\zeta^{-1}P(\gamma_+,\gamma_-)\subset \widetilde{M}(\Lambda,R)
\end{equation} (see the shaded area in Figure~\ref{lrf3}) is
a fundamental domain for the $\Lambda$-action on $\widetilde{M}(\Lambda,R)$. It follows that the compact translation surface ${M}(\Lambda,R)$ can be represented as the union of two
identical tori glued along a horizontal slit of length $2R$ as in Figure~\ref{surfcomlens}.

Let $p:\widetilde{M}(\Lambda,R)\to M(\Lambda,R)$ denote the
covering map and consider vectors $\gamma_+^+$,
$\gamma_+^-$, $\gamma_-^+$, $\gamma_-^-$ in
$\widetilde{M}(\Lambda,R)$ as in Figure~\ref{lrf3}. Then
$\zeta(\gamma_+^+)=\zeta(\gamma_+^-)=\gamma_+$ and
$\zeta(\gamma_-^+)=\zeta(\gamma_-^-)=\gamma_-$. We will denote
also by $\gamma_+^+$, $\gamma_+^-$, $\gamma_-^+$, $\gamma_-^-$ the
corresponding oriented curves in $\widetilde{M}(\Lambda,R)$. The
projections $p(\gamma_+^+)$, $p(\gamma_+^-)$, $p(\gamma_-^+)$,
$p(\gamma_-^-)$ are oriented loops in ${M}(\Lambda,R)$ whose
homology classes generate the group $H_1({M}(\Lambda,R),\Z)$. Define
\[\gamma_1:=[p(\gamma_+^+)]+[p(\gamma_+^-)],\
\gamma_2:=[p(\gamma_-^+)]+[p(\gamma_-^-)]\in H_1({M}(\Lambda,R),\Z),\] see Figure~\ref{surfcomlens}.
Now use the group isomorphism
$\Z^2\ni(m,n)\mapsto m\gamma_++n\gamma_-\in\Lambda$ to convert the $\Lambda$-action on $F(\Lambda,R)$ and $\widetilde{M}(\Lambda,R)$ into a $\Z^2$-action. By \eqref{def:act} we have
\begin{equation}\label{def:actzety}
\zeta((m,n)\cdot
\widetilde{x})=(m,n)\cdot\zeta(\widetilde{x})\quad\text{ for all
}\ (m,n)\in\Z^2\text{ and }\ \widetilde{x}\in
\widetilde{M}(\Lambda,R).
\end{equation}

\begin{figure}[h]
\includegraphics[width=0.5\textwidth]{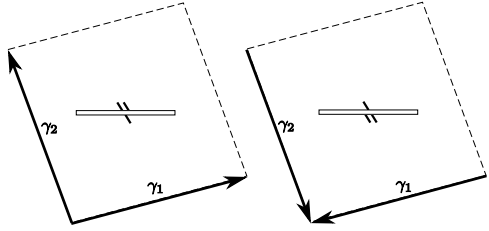}
\caption{The surface $M(\Lambda,R)$ \label{surfcomlens}}
\end{figure}
\begin{lemma}
The translation surface $\widetilde{M}(\Lambda,R)$  is the
$\Z^2$-cover of the compact translation surface ${M}(\Lambda,R)$
defined by $\gamma=(\gamma_2,-\gamma_1)\in H_1(M,\Z)^2$.
\end{lemma}
\begin{proof}
In view of Remark~\ref{rem:cover}, we need to choose a finite set
of oriented loops in $M(\Lambda,R)$ whose homology classes
generate $H_1({M}(\Lambda,R),\Z)$ and show, that for every such
loop $\sigma:[0,1]\to M(\Lambda,R)$ any its lift
$\widetilde{\sigma}:[0,1]\to \widetilde{M}(\Lambda,R)$ fulfills
\[\widetilde{\sigma}(1)=\big(\langle[\sigma],\gamma_2\rangle,
-\langle[\sigma],\gamma_1\rangle \big)\cdot \widetilde{\sigma}(0).\] Of course, we will deal with the loops $p(\gamma_+^+)$, $p(\gamma_+^-)$,
$p(\gamma_-^+)$, $p(\gamma_-^-)$ whose lifts $\gamma_+^+$, $\gamma_+^-$, $\gamma_-^+$, $\gamma_-^-$ satisfy
\[\gamma_+^\pm(1)=(1,0)\cdot\gamma_+^\pm(0),\quad \gamma_-^\pm(1)=(0,1)\cdot\gamma_-^\pm(0), \] see Figure~\ref{lrf3}.
On the other hand,
\[\big(\langle[p(\gamma_+^\pm)],\gamma_2\rangle,
-\langle[p(\gamma_+^\pm)],\gamma_1\rangle \big)=(1,0)\text{ and
}\big(\langle[p(\gamma_-^\pm)],\gamma_2\rangle,
-\langle[p(\gamma_-^\pm)],\gamma_1\rangle \big)=(0,1),\] which
completes the proof.
\end{proof}

\section{Geometric step of the proof of Theorem~\ref{thm:mainflat}}
The proof of Theorem~\ref{thm:mainflat} can be divided into two parts. The  first part
relies on Teichm\"uller dynamics,  the Kontsevich-Zorich cocycle
and a bounded deviation phenomenon. The corresponding statement, Theorem~\ref{thm:existxi} below, ensures the existence (for a.e.\ $\Lambda$) of a
non-trivial homology class $\xi\in H_1(M(\Lambda,R),\R)$ such that, roughly speaking, the intersection number of $\xi$ with arbitrary vertical orbit
segment on $M(\Lambda,R)$ is uniformly bounded. This is technically the most involved part, so we postpone the proof of Theorem~\ref{thm:existxi} together
with the necessary background until Section~\ref{Teich:sec}.

The second part, Theorem~\ref{thm:existband} below, is geometric. We use the homology class $\xi\in H_1(M(\Lambda,R),\R)$ to identify the direction  of the
bands trapping vertical light rays in $F(\Lambda,R)$. More precisely, we show that this direction is given by the vector
\[\bar{v}(\Lambda,\xi):=\langle\gamma_2,\xi\rangle\gamma_+-\langle\gamma_1,\xi\rangle\gamma_-\in\R^2,\]
for vectors $\gamma_+$, $\gamma_-$ and  homology classes $\gamma_1$, $\gamma_2$ defined in Section~\ref{sec:red}.

In order to formulate Theorem \ref{thm:existxi} and Theorem \ref{thm:existband} we need auxiliary notation.
\begin{notation}
Let $(\widetilde{M}_\gamma,\widetilde{\omega}_\gamma)$ be a
$\Z^d$-cover of a compact translation surface $(M,\omega)$.
Denote by $M^+_\omega$ the set of points $x\in M$
such that  the positive semi-orbit $(\varphi^v_t(x))_{t\geq 0}$ on
$(M,\omega)$ is well defined. Let $D\subset \widetilde{M}_\gamma$
be a bounded fundamental domain of the cover such that the
interior of $D$ is path-connected and the boundary of $D$ is a
finite union of intervals. For every $x\in M^+_\omega$ and $t>0$
define the element $\sigma^{\omega}_t(x) \in H_1(M,\Z)$ as
the homology class of the loop formed by the segment of the vertical orbit
of $x$ from $x$ to $\varphi^v_t(x)$ closed up by the shortest curve
joining $\varphi^v_t(x)$ with $x$ that does not cross
$p^{-1}(\partial D)$.
\end{notation}

Following \eqref{def:domainD}, we denote fundamental domains for surfaces $M(\Lambda,R)$
and their covers $\widetilde{M}(\Lambda,R)$ by $D$.

\begin{theorem}\label{thm:existxi}
Let $\omega$ be the Abelian differential on $M$ determining the translation structure on  $M(\Lambda,R)$. Then for every $R>0$ and $\mu_{\mathscr{L}}$-a.e.\ lattice $\Lambda\in\mathscr{L}$ there
exists $0\neq \xi\in\R\gamma_1+\R\gamma_2\subset H_1(M,\R)$ and
$C>0$ such that
\[|\langle\sigma^\omega_t(x),\xi\rangle|\leq C\text{ for every }x\in M^+_\omega\text{ and }t>0.\]
\end{theorem}

\begin{theorem}\label{thm:existband}
Suppose that $R>0$ and
$\Lambda\in\mathscr{L}$ satisfy \eqref{cond:adm} and $M(\Lambda,R)=(M,\omega)$.
Further assume that there is a non-zero homology class
$\xi\in\R\gamma_1+\R\gamma_2$ and $C>0$ such that
\[|\langle\sigma^\omega_t(x),\xi\rangle|\leq C\text{ for every }x\in
M^+_\omega\text{ and }t>0.\] If the surface $M(\Lambda,R)$ has no vertical saddle connection, i.e.\  there is no  vertical orbit segment that connect
singular points, then there exists $\overline{C}>0$ such  that every vertical orbit on $F(\Lambda,R)$ is trapped in an infinite band in direction
$\bar{v}(\Lambda,\xi)$.  Furthermore the width of that band is bounded by $\overline{C}$.
\end{theorem}

Since condition \eqref{cond:adm} is satisfied
for every $R>0$ and $\mu_{\mathscr{L}}$-a.e.\  lattice
$\Lambda\in\mathscr{L}$
and $M(\Lambda,R)$ has no vertical saddle connection
(see Remark~\ref{rem:saddle}), Theorem~\ref{thm:mainflat} is an
obvious consequence of the above two theorems.

\medskip

In the remainder of this section we will prove
Theorem~\ref{thm:existband}.  The proof will be preceded by a
series of useful observations. We shall postpone the proof of
Theorem~\ref{thm:existxi}  until Section~\ref{Teich:sec}.

\medskip

Let $(M,\omega)$ be a compact translation surface and let  m$(\widetilde{M}_\gamma,\widetilde{\omega}_\gamma)$ be its $\Z^d$-cover given by
$\gamma=(\gamma_1,\ldots,\gamma_d)\in H_1(M,\Z)^d$. Denote by $(\widetilde{\varphi}^{v}_t)_{t\in\R}$ the vertical flow on the $\Z^d$-cover
$(\widetilde{M}_\gamma,\widetilde{\omega}_\gamma)$.
Let us consider  the map
$\bar{m}:\widetilde{M}_\gamma\to\Z^d$  such that
$\bar{m}=\bar{m}(\widetilde{x})\in\Z^d$ is the unique element with
$\widetilde{x}\in \bar{m}\cdot D$.
\begin{lemma}\label{lem:posorb}
Let $\widetilde{x}\in \widetilde{M}_\gamma$ and $x=p(\widetilde{x})\in M^+_\omega$. Let $\xi\in H_1(M,\R)$ be an element such that
$\xi=\sum_{i=1}^da_i\gamma_i$. Then for every $t>0$ we have
\[(a_1,\ldots,a_d)\cdot
(\bar{m}(\widetilde{\varphi}^{v}_t\widetilde{x})-\bar{m}(\widetilde{x}))=\langle\sigma^\omega_t(x),\xi\rangle.\]
\end{lemma}

\begin{proof}
Let
$\bar{m}=\bar{m}(\widetilde{\varphi}^{v}_t\widetilde{x})-\bar{m}(\widetilde{x})$.
Then both
$\widetilde{\varphi}^{v}_t\widetilde{x},\ \bar{m}\cdot\widetilde{x}\in
(\bar{m}(\widetilde{x})+\bar{m})\cdot D$. Let us consider the
curve $\sigma_t(\widetilde{x})$ in $\widetilde{M}_\gamma$ which is
formed by the segment of the vertical orbit of $\widetilde{x}$
from $\widetilde{x}$
 to
$\widetilde{\varphi}^v_t(\widetilde{x})$ together with the shortest  curve in $(\bar{m}(\widetilde{x})+\bar{m})\cdot D$ joining
$\widetilde{\varphi}^v_t(\widetilde{x})$ with $\bar{m}\cdot\widetilde{x}$. By definition, $\sigma^\omega_t(x)=[p\circ \sigma_t(\widetilde{x})]$. Since
$(\gamma_1,\ldots,\gamma_d)$ determines the cover, it follows that the beginning $\widetilde{x}$  and the end $\bar{m}\cdot\widetilde{x}$ of the curve
$\sigma_t(\widetilde{x})$ satisfy
\[\bar{m}\cdot\widetilde{x}=(\langle \sigma^\omega_t(x),\gamma_1 \rangle,\ldots,\langle \sigma^\omega_t(x),\gamma_d\rangle )\cdot\widetilde{x}.\]
Since the $\Z^d$-action on $\widetilde{M}_\gamma$ is free, it follows that
\[\bar{m}=(\langle \sigma^\omega_t(x),\gamma_1 \rangle,\ldots,\langle \sigma^\omega_t(x),\gamma_d\rangle ).\]
Consequently,
\[(a_1,\ldots,a_d)\cdot
\bar{m}=\langle\sigma^\omega_t(x),\sum_{i=1}^da_i\gamma_i\rangle=\langle\sigma^\omega_t(x),\xi\rangle.\]
\end{proof}

Let us consider  the functions $\widetilde{m}:\widetilde{M}(\Lambda,R)\to\Z^2$ and $\widehat{m}:F(\Lambda,R)\to\Z^2$ such that
$\widetilde{m}(\widetilde{x}), \widehat{m}(\widehat{x})\in\Z^2$ are the unique elements with $\widetilde{x}\in \widetilde{m}(\widetilde{x})\cdot D$ and
$\widehat{x}\in \widehat{m}(\widehat{x})\cdot P(\gamma_+,\gamma_-)$ for all $\widetilde{x}\in \widetilde{M}(\Lambda,R)$ and $\widehat{x}\in F(\Lambda,R)$.
In view of \eqref{def:domainD} and \eqref{def:actzety}, we have
\begin{equation}\label{eq:twoonm}
\widehat{m}(\zeta\widetilde{x})=\widetilde{m}(\widetilde{x})
\quad\text{ for every }\quad
\widetilde{x}\in\widetilde{M}(\Lambda,R)
\end{equation}
and, by definition, for every $\widehat{x}\in F(\Lambda,R)$ we have
\begin{equation}\label{eq:gamdecomp}
\widehat{x}=\widehat{m}_1(\widehat{x})\gamma_++\widehat{m}_2(\widehat{x})\gamma_-+\gamma_+y_1+\gamma_-y_2
\quad\text{ for some }y_1,y_2\in[-1/2,1/2).
\end{equation}

\begin{remark}
For every $\widehat{x}\in F(\Lambda,R)$ let
$\mathcal{O}(\widehat{x})$ be the light ray orbit passing through
$\widehat{x}$. If $\mathcal{O}(\widehat{x})$  is a \emph{regular
orbit}, i.e.\ $\mathcal{O}(\widehat{x})$ does not pass through the
ends of any slit in $F(\Lambda,R)$ then, by the definition of the
light ray flow $(\widetilde{\varphi}_t)_{t\in\R}$ on
$\widetilde{M}(\Lambda,R)$, we have
\begin{equation}\label{eq:Oreg}
\mathcal{O}(\widehat{x})=\zeta\{\widetilde{\varphi}_t\widetilde{x}:t\in\R\}\text{
for any }\widetilde{x}\in\widetilde{M}(\Lambda,R)\text{ with
}\zeta\widetilde{x}=\widehat{x}.
\end{equation}

Now suppose that an orbit $\mathcal{O}(\widehat{x})$ passes
through the ends of slits once and let $\widehat{x}$ be such end.
For small $\vep>0$ let
$\widehat{x}_{-\vep}:=\widehat{x}-(0,\vep)\in F(\Lambda,R)$ and
$\widehat{x}_{\vep}:=\widehat{x}+(0,\vep)\in F(\Lambda,R)$. Next
choose $\widetilde{x}_{-\vep},\widetilde{x}_{\vep}\in
F_+(\Lambda,R)\subset M(\Lambda,R)$ such that
$\zeta\widetilde{x}_{-\vep}=\widehat{x}_{-\vep}$ and
$\zeta\widetilde{x}_{\vep}=\widehat{x}_{\vep}$. Then
\begin{equation}\label{eq:Oirreg}
\mathcal{O}(\widehat{x})=\bigcup_{\vep>0}\zeta\{\widetilde{\varphi}_t\widetilde{x}_\vep:t\geq
0\}\cup
\zeta\{\widetilde{\varphi}_{-t}\widetilde{x}_{-\vep}:t\geq
0\}\cup\{\widehat{x}\}.
\end{equation}
\end{remark}

Let us consider the map $\rho:\widetilde{M}(\Lambda,R)\to\widetilde{M}(\Lambda,R)$ defined by
\[\rho( \widetilde{x})=
  \begin{cases}
   r_\pi\circ \zeta_-^{-1}\circ \zeta_+(\widetilde{x}) & \text{ if }\widetilde{x}\in F_+(\Lambda,R), \\
   \zeta_+^{-1}\circ \zeta_-\circ r_\pi^{-1}(\widetilde{x}) & \text{ if }\widetilde{x}\in r_\pi F_-(\Lambda,R).
  \end{cases}
\]
\begin{lemma}
The map $\rho:\widetilde{M}(\Lambda,R)\to\widetilde{M}(\Lambda,R)$ is an involution
with the properties
\begin{equation}\label{lemeq:rho}
\zeta\circ\rho=\zeta\quad\text{ and }\quad\rho^*(\omega)=-\omega.
\end{equation}
\end{lemma}
\begin{proof}
By definition, $\rho$ is an involution that maps $F_+(\Lambda,R)$
on  $r_\pi F_-(\Lambda,R)$ and $r_\pi F_-(\Lambda,R)$ on
$F_+(\Lambda,R)$. Comparing this with the definition of $\zeta$
(see \eqref{def:zeta}) immediately gives $\zeta\circ\rho=\zeta$.

Recall that the form $\omega$ is given by $dz$ on both
$F_+(\Lambda,R)$  and $r_\pi F_-(\Lambda,R)$. The map
$\zeta_-^{-1}\circ \zeta_+:F_+(\Lambda,R)\to F_-(\Lambda,R)$ is a
bijection that sends $dz$ to $dz$. Moreover, the rotation
$r_\pi:F_-(\Lambda,R)\to r_\pi F_-(\Lambda,R)$ sends $dz$ to
$-dz$. Since $\rho$ is the composition (on pieces) the above maps
or their inverses, it follows that  $\rho^*(\omega)=-\omega$.
\end{proof}

Note that, by the definition of directional flows,
$\rho^*(\omega)=-\omega$  immediately implies
\begin{equation}\label{eq:revers}
\rho\circ\widetilde{\varphi}^v_t(\widetilde{x})=\widetilde{\varphi}^v_{-t}\circ\rho(\widetilde{x})
\end{equation}
for all $\widetilde{x}\in \widetilde{M}(\Lambda,R)$ and $t\in \R$ for  which both sides of \eqref{eq:revers} are well defined.



\begin{proof}[Proof of Theorem~\ref{thm:existband}]
Let $\mathcal{O}$ be a trajectory of light in $F(\Lambda,R)$.
Since $M(\Lambda,R)$ has no vertical saddle connection,
$\mathcal{O}$ passes through the ends of a slit in $F(\Lambda,R)$
at most once.


Suppose first that $\mathcal{O}=\mathcal{O}(\widehat{x})$ is a
regular  orbit for some $\widehat{x}\in F(\Lambda,R)$. Let
$\widetilde{x}\in M(\Lambda,R)$ be such that
$\zeta\widetilde{x}=\widehat{x}\in\mathcal{O}$. In view of
\eqref{eq:Oreg}, \eqref{eq:revers} and \eqref{lemeq:rho},
\[\mathcal{O}(\widehat{x})=\zeta\{\widetilde{\varphi}^v_t\widetilde{x}:t\in\R\}=\zeta\{\widetilde{\varphi}^v_t\widetilde{x}:t\geq 0\}\cup  \zeta\{\widetilde{\varphi}^v_t\rho\widetilde{x}:t\geq 0\}\]
and $p(\widetilde{x}), p(\rho\widetilde{x})\in M^+_\omega$.

As $\xi\in\R\gamma_1+\R\gamma_2$ and $\langle\gamma_1, \gamma_2\rangle=2$, we have $\xi=a\gamma_2-b\gamma_1$ with
\[a=\frac{\langle\gamma_1,\xi\rangle}{\langle\gamma_1,\gamma_2\rangle}=\frac12\langle\gamma_1,\xi\rangle\quad\text{   and }
\quad b=-\frac{\langle\gamma_2,\xi\rangle}{\langle\gamma_2,\gamma_1\rangle}=\frac12\langle\gamma_2,\xi\rangle.\] Since the pair $(\gamma_2,-\gamma_1)$
defines the cover, by Lemma~\ref{lem:posorb},
\[\big|(a,b)\cdot(\widetilde{m}(\widetilde{\varphi}^v_t\widetilde{x})-\widetilde{m}(\widetilde{x}))\big|=
|\langle\sigma^\omega_t(p(\widetilde{x})),\xi\rangle|\leq C\quad
\text{ for }\quad t\geq 0,
\]
moreover, by \eqref{eq:twoonm}, \eqref{lemeq:rho} and \eqref{eq:revers}, for $t<0$ we have
\begin{align*}\big|(a,b)\cdot(\widetilde{m}(\widetilde{\varphi}^v_t\widetilde{x})-\widetilde{m}(\widetilde{x}))\big|&=
\big|(a,b)\cdot(\widetilde{m}(\widetilde{\varphi}^v_{-t}\rho\widetilde{x})-\widetilde{m}(\rho\widetilde{x}))\big|\\&=
|\langle\sigma^\omega_{-t}(p(\rho\widetilde{x})),\xi\rangle|\leq
C.
\end{align*}
For every $t\in\R$ let
$\widehat{x}_t=\zeta\big(\widetilde{\varphi}^v_t\widetilde{x}\big)$
and
\[\bar{m}^t:=\widetilde{m}(\widetilde{\varphi}^v_t\widetilde{x})-\widetilde{m}(\widetilde{x})=\widehat{m}(\widehat{x}_t)-\widehat{m}(\widehat{x}),\]
see \eqref{eq:twoonm}. By \eqref{eq:gamdecomp}, this yields
\[\widehat{x}_t=\widehat{x}+\bar{m}_1^t\gamma_++\bar{m}_2^t\gamma_-+\gamma_+y_1+\gamma_-y_2\]
with $y_1,y_2\in[-1,1]$. Therefore
\[\widehat{x}_t-\widehat{x}-\frac{(b,-a)\cdot \bar{m}^t}{a^2+b^2}(b\gamma_+-a\gamma_-)=\frac{(a,b)\cdot\bar{m}^t}{a^2+b^2}(a\gamma_++b\gamma_-)+\gamma_+y_1+\gamma_-y_2.\]
Since $|(a,b)\cdot\bar{m}^t|\leq C$ for every $t\in\R$ and
\[b\gamma_+-a\gamma_-=\frac{1}{\langle\gamma_1,\gamma_2\rangle}(\langle\gamma_2,\xi\rangle\gamma_+-\langle\gamma_1,\xi\rangle\gamma_-)=\frac{\bar{v}(\Lambda,\xi)}{\langle\gamma_1,\gamma_2\rangle},\]
it follows that
\[\operatorname{dist}\big(\widehat{x}_t-\widehat{x},\R\, \bar{v}(\Lambda,\xi)\big)\leq\overline{C}:=\Big(\frac{C}{\sqrt{a^2+b^2}}+1\Big)(\|\gamma_+\|+\|\gamma_-\|).\]

Finally suppose that $\mathcal{O}=\mathcal{O}(\widehat{x})$ is not
regular and $\widehat{x}$  is the end of a slit.
Then, in view of \eqref{eq:Oirreg},
\eqref{eq:revers} and \eqref{lemeq:rho},
\[\mathcal{O}=\bigcup_{\vep>0}\zeta\{\widetilde{\varphi}^v_t\widetilde{x}_\vep:t\geq 0\}\cup \zeta\{\widetilde{\varphi}^v_t\rho\widetilde{x}_{-\vep}:t\geq 0\}\cup\{\widehat{x}\}\]
and $p(\widetilde{x}_{-\vep}), p(\rho\widetilde{x}_{-\vep})\in
M^+_\omega$.  Now the rest of the proof runs as in the regular
case.
\end{proof}

\section{Moduli space, Teichm\"uller flow and  Kontsevich-Zorich cocycle}\label{Teich:sec}
In this section we give a brief overview of the Teichm\"uller flow
and the Kontsevich-Zorich  cocycle. For further background
material we refer the reader to
\cite{For-coheq,For-dev,YoLN,ZoFlat}.

Given  a connected compact orientable surface $M$ denote by
$\operatorname{Diff}^+(M)$  the group of orientation-preserving
homeomorphisms of $M$. We will denote by $\mathcal{M}(M)$
($\mathcal{M}_a(M)$) the {\em moduli space} of (area $a>0$)
Abelian differentials, that is the space of orbits of the natural
action of $\operatorname{Diff}^+(M)$ on the space of (area $a>0$)
Abelian differentials on $M$. Nevertheless $\mathcal{M}_a(M)$ can be
identified with $\mathcal{M}_1(M)$, by rescaling Abelian
differentials with the factor $1/\sqrt{a}$.

The group $SL(2,\R)$ acts naturally on the space of Abelian differentials on $M$ and $\mathcal{M}(M)$ as follows: given a translation structure $\omega$,
consider the charts given by  local primitives of the holomorphic $1$-form. The new charts defined by postcomposition of these charts with an element of
$SL(2,\R)$ define a new complex structure and a new differential which is holomorphic with respect to this new complex structure, thus a new translation
structure. We denote by $g\cdot \omega$ the translation structure  on $M$ obtained acting by $g \in SL(2,\R)$ on a translation structure $\omega$ on $M$.
Since $\mathcal{M}_a(M)$ is $SL(2,\R)$-invariant, we restrict the action $SL(2,\R)$ to $\mathcal{M}_a(M)$.

Using the Iwasawa NAK decomposition, every element of $SL(2,\R)$
has a unique decomposition $h_sg_tr_\theta$, where
\begin{equation*}
h_s= \left(\begin{array}{cc} 1 & 0 \\ s &
1 \end{array}\right), \quad
g_t= \left(\begin{array}{cc} e^t & 0 \\ 0 & e^{-t}
\end{array}\right), \quad  r_\theta= \left(\begin{array}{cc} \cos
\theta & -\sin \theta \\ \sin \theta & \cos\theta
\end{array}\right).
\end{equation*}
\begin{remark}
Since the action of  $h_sg_t$ rescales the vertical vector $(0,1)$
 by the factor $e^{-t}$, the vertical flow on
$(M,h_sg_t\cdot\omega)$ coincides with the linear time change
 by the factor $e^{-t}$ of the vertical flow on
$(M,\omega)$. It follows that
\begin{equation}\label{eq:changesigma}
\sigma_T^{h_sg_t\cdot\omega}(x)=\sigma_{e^tT}^{\omega}(x).
\end{equation}
\end{remark}

The restriction of $SL(2,\R)$ on $\mathcal{M}_1(M)$ to the
diagonal subgroup $(g_t)_{t\in\R}$ is called the {\em
Teichm\"uller flow} and we will denote this flow by
$(g_t)_{t\in\R}$.

Let  $M$ be a surface of genus $g$ and let $m$ be the number of
zeros of $\omega$.  If $\alpha_i$, $1\leq i\leq m$ are degrees of
all zeros, one has $2g-2 = \sum_{i=1}^m\alpha_i$. Let us denote by
$\alpha=(\alpha_1,\ldots,\alpha_m)$ and $\mathcal{M} (\alpha)$ the
\emph{stratum} consisting of all $(M, \omega)$ such that $\omega$
has $m$ zeros of degrees $\alpha_1, \ldots, \alpha_m$. Then the
normalized stratum $\mathcal{M}_1 (\alpha) = \mathcal{M} (\alpha)
\cap \mathcal{M}_1(M)$ is also $SL(2,\R)$-invariant.

 The {\em Kontsevich-Zorich cocycle}
$(G^{KZ}_t)_{t\in\R}$ is the quotient of the trivial action
$(g_t\times Id)_{t\in\R}$ on the product of the space of area one
Abelian differentials on $M$ with $H_1(M,\R)$ by the action of the
group $\operatorname{Diff}^+(M)$.  Elements of
$\operatorname{Diff}^+(M)$ act on the fiber $H_1(M,\R)$ by induced
maps. The cocycle $(G^{KZ}_t)_{t\in\R}$ acts on the homology
vector bundle $\mathcal{H}_1(M,\R)$
 over the Teichm\"uller flow $(g_t)_{t \in \R}$ on the moduli space $\mathcal{M}_1(M)$.

Clearly the fibers of the  bundle $\mathcal{H}_1(M,\R)$ can be
identified with $H_1(M,\R)$.  The algebraic intersection
number furnishes $H_1(M,\R)$ with a symplectic structure.
This symplectic structure is invariant under the action of
the mapping-class group and hence invariant under the action of
$SL(2,\R)$.

The standard definition of the KZ-cocycle is based on the cohomological
bundle $\mathcal{H}^1(M,\R)$. Each fiber of this bundle (identified with
${H}^1(M,\R)$) is endowed with a natural norm, called the
\emph{Hodge norm}, see for example \cite{For-dev} for definition.
The identification of the homological and cohomological
bundle and the corresponding KZ-cocycles is established by
Poincar\'e duality $\mathcal{P}:H_1(M,\R)\to H^1(M,\R)$.
Via Poincar\'e duality, the Hodge norms induce norms on the fibers of
$\mathcal{H}_1(M,\R)$. The norm on the fiber $H_1(M,\R)$
over $\omega\in\mathcal{M}(M)$ is denoted by
$\|\,\cdot\,\|_\omega$ and will be called \emph{Hodge norm} as well.

Let $\nu$ be a probability measure on $\mathcal{M}_1(M)$ which  is
$(g_t)_{t\in\R}$-invariant and ergodic. Suppose that
$\mathcal{M}_{\nu}\subset\mathcal{M}_1(M)$ is a
$(g_t)_{t\in\R}$-invariant set with full $\nu$-measure and
$V\subset H_1(M,\R)$ is a symplectic subspace, i.e.\ the symplectic form
restricted to $V$ is non-degenerate. Moreover,
assume that $V$ is invariant for the action induced from
$(g_t)_{t\in\R}$-action on $\mathcal{M}_\nu$. Then $V$ defines a
subbundle, denoted by $\mathcal{V}$, of the bundle
$\mathcal{H}_1(M,\R)$ over $\mathcal{M}_\nu$ for which the fibers
are identified with $V$.

Let us consider the KZ-cocycle $(G_t^{\mathcal{V}})_{t\in\R}$
restricted to $V$. Since the measure $\nu$ is ergodic, by
Oseledets'  theorem, there exists Lyapunov exponents of
$(G_t^{\mathcal{V}})_{t\in\R}$ with respect to the measure $\nu$.
As the action of the Kontsevich-Zorich cocycle is symplectic, its
Lyapunov exponents with respect to the measure $\nu$ are:
\[\lambda^{\mathcal{V}}_1>\lambda^{\mathcal{V}}_2>\ldots>\lambda^{\mathcal{V}}_s\geq-
\lambda^{\mathcal{V}}_s>\ldots>-\lambda^{\mathcal{V}}_2>-\lambda^{\mathcal{V}}_1\]
and for $\nu$-a.e.\ $\omega\in \mathcal{M}_\nu$ there is a splitting
$V=\bigoplus_{i=-s}^sV_i(\omega)$ (if $\lambda^{\mathcal{V}}_s=0$ then
$V_{-s}(\omega)=V_{s}(\omega)$) such that for
any $\xi\in V_i(\omega)$ we have
\begin{equation*}
\lim_{t\to\infty}\frac{1}{t}\log\|\xi\|_{g_t\omega}=\lambda^{\mathcal{V}}_i.
\end{equation*}
It follows that $V$ has a direct splitting
\[ V=E_{\omega}^+\oplus E_{\omega}^0\oplus E_{\omega}^-\quad\text{ for $\nu$-a.e.\ } \omega\in \mathcal{M}_\nu,\]
into unstable, central and stable subspaces
\begin{align*}
E_{\omega}^+&=\Big\{\xi\in V:\lim_{t\to+\infty}\frac{1}{t}\log\|\xi\|_{g_{-t}\omega}<0\Big\},
\label{stabledef}\\
E_{\omega}^0&=\Big\{\xi\in V:\lim_{t\to\infty}\frac{1}{t}\log\|\xi\|_{g_{t}\omega}=0\Big\},\nonumber
\\
E_{\omega}^-&=\Big\{\xi\in
V:\lim_{t\to+\infty}\frac{1}{t}\log\|\xi\|_{g_{t}\omega}<0\Big\}.\nonumber
\end{align*}
The dimension of the stable and unstable subspace is equal to the
number of positive Lyapunov exponents of $(G^{\mathcal{V}}_t)_{t\in\R}$.

The following theorem, crucial to the proof of
Theorem~\ref{thm:existxi},  is based on the phenomenon of bounded
deviation discovered by Zorich in his seminal papers \cite{Zor1}
and \cite{Zor2}. Its proof can be found in  \cite{DHL} (see
Theorem~2) and \cite{Fr-Ulc:nonerg} (see Theorem~4.2 - in the
cohomological setting).
\begin{theorem}\label{thm:existbound}
For $\nu$-a.e.\ $\omega\in \mathcal{M}_\nu$ there exists $C>0$
such that for every $\xi\in E_{\omega}^-$, $x\in M^+_\omega$ and
$t>0$ we have $|\langle\sigma^{\omega}_t(x),\xi\rangle|\leq
C\|\xi\|_\omega$.
\end{theorem}

\section{Branched $2$-covers of the torus and the proof of Theorem~\ref{thm:existxi}}
Let us consider the standard torus $\T^2=\R^2/\Z^2$ with two
different marked points $u_1,u_2\in\T$. Denote by $M$ the $2$-cover
of $\T^2$ branched at $u_1$ and $u_2$, see
Figure~\ref{surfcomcover}. Denote by $p:M\to\T^2$ the covering
map. Then the deck group consists of $id$ and the involution
$\tau:M\to M$ exchanging the points in fibers over
$\T^2\setminus\{u_1,u_2\}$ and fixing the points $u_1,u_2$.
\begin{figure}[h]
\includegraphics[width=0.8\textwidth]{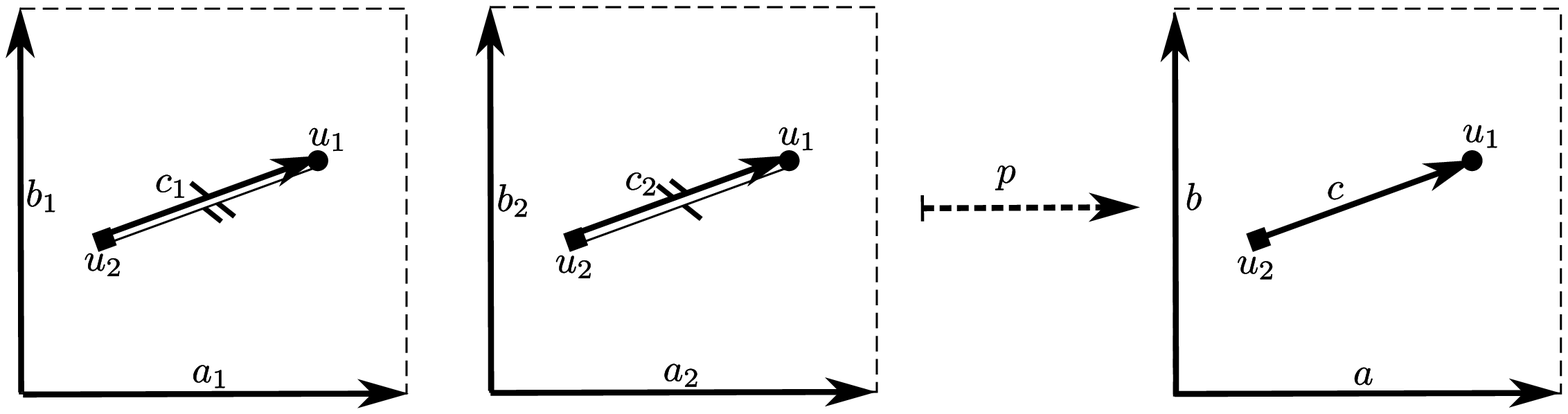}
\caption{Branched $2$-cover $p:M\to\T^2$ \label{surfcomcover}}
\end{figure}
Let us consider the set
\begin{align*}\mathcal{M}=&\{(M,\omega)\in\mathcal{M}_2(1,1):\tau^*\omega=\omega\}\\
=&\{(M,\omega)\in\mathcal{M}_2(1,1):\omega=p^*\omega_0\text{ for some }\omega_0\in\mathcal{M}_1(0,0)\}.
\end{align*}
Then $\mathcal{M}$ is an $SL(2,\R)$-invariant subset of
$\mathcal{M}_2(1,1)$  which is a $2$-cover of the stratum
$\mathcal{M}_1(0,0)$.  Therefore $\mathcal{M}$ has a natural
$SL(2,\R)$-invariant $(g_t)_{t\in\R}$-ergodic measure
$\nu_{\mathcal{M}}$ which is the pullback by the covering map of
the canonical measure on the stratum $\mathcal{M}_1(0,0)$.

Let us consider the orthogonal (symplectic) decomposition $H_1(M,\R)=V\oplus V^\perp$  with
\[V=\{\xi\in H_1(M,\R):\tau_*\xi=-\xi\}\text{ and }V^\perp=\{\xi\in H_1(M,\R):\tau_*\xi=\xi\}.\]
Then $V$ and $V^\perp$ are $2$-dimensional symplectic subspaces and  $\gamma_1:=a_1-a_2$, $\gamma_2:=b_1-b_2\in H_1(M,\Z)$ establish a basis of $V$. Since
$V=\ker p_*$ and $p$ is $SL(2,\R)$-equivariant, $V$ is invariant under the induced $SL(2,\R)$-action. Therefore, $V$ defines a subbundle over $\mathcal{M}$
and we can consider the restricted KZ cocycle $(G_t^{\mathcal{V}})_{t\in\R}$ on $\mathcal{V}$. As it was shown in \cite{Bain} by Bainbridge, the Lyapunov
exponents of every ergodic $SL(2,\R)$-invariant measure on $\mathcal{M}_2(1,1)$ are equal to $1$, $1/2$, $-1/2$, and  $-1$. It follows that the Lyapunov
exponents of $(G_t^{\mathcal{V}})_{t\in\R}$ are $1/2$ and $-1/2$. Therefore, for $\nu_\mathcal{M}$-a.e.\ $\omega\in\mathcal{M}$ the stable subspace
$E_{\omega}^-\subset V$ is one-dimensional (so non-trivial). We now apply Theorem~\ref{thm:existband} to the measure $\nu_{\mathcal{M}}$ to obtain the
following result.

\begin{lemma}\label{lem:existbound}
For $\nu_{\mathcal{M}}$-a.e.\ surface $(M,\omega)$
\begin{equation}\label{asstion-thm}
\text{there exist }\xi\in\!V\!\setminus\!\{0\}, C>0\text{  such that } |\langle\sigma^\omega_t(x),\xi\rangle|\leq C\text{ for all }x\in \!M^+_\omega, t>0.
\end{equation}
\end{lemma}
\begin{proof}[Proof of Theorem~\ref{thm:existxi}]
We begin the proof by describing a local product structure on
$\mathcal{M}$ that will help us to deduce
Theorem~\ref{thm:existxi} directly from
Lemma~\ref{lem:existbound}.

Since $\mathcal{M} \subset\mathcal{M}_2(1,1)$ is a $2$-cover of
the stratum $\mathcal{M}_1(0,0)$, it is a $5$ dimensional
manifold. Each element $(M,\omega)$ of $\mathcal{M}$ is a union of
two identical tori glued along a slit, see
Figure~\ref{surfcomgenform}.
\begin{figure}[h]
\includegraphics[width=0.5\textwidth]{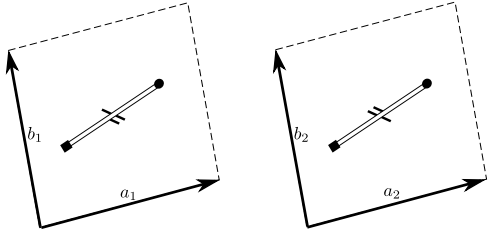}
\caption{Typical surface in $\mathcal{M}$ \label{surfcomgenform}}
\end{figure}

For fixed $R>0$ let $\Upsilon:SL(2,\R)/SL(2,\Z)\times\R^2\to
\mathcal{M}$  be the map
\[\Upsilon(\Lambda,t,s)=(h_sg_t)\cdot M(\Lambda,R).\]
This map is a local diffeomorphism whose image consists of surfaces in  $\mathcal{M}$ whose slits are not vertical. Indeed, local inverses of $\Upsilon$ are given as follows. To start represent every surface $(M,\omega)$ in $\mathcal{M}$ as the union of two identical tori glued along a slit, see Figure~\ref{surfcomgenform}. Assume that the slit is not vertical. Then transform $(M,\omega)$ by $h_{-s}$ so that the slit of the resulting surface
$(M,h_{-s}\cdot\omega)$ is horizontal. Next transform $(M,h_{-s}\cdot\omega)$ by $g_{-t}$ so that the length of the horizontal slit for the resulting
surface $(M,(g_{-t}h_{-s})\cdot\omega)$ is $2R$. Therefore, $(M,(g_{-t}h_{-s})\cdot\omega)=M(\Lambda,R)$ for some $\Lambda\in\mathscr{L}$, so
$(M,\omega)=\Upsilon(\Lambda,s,t)$.

Since $\nu_{\mathcal{M}}$ is a smooth measure on $\mathcal{M}$ and $\Upsilon$ is a local diffeomorphism, the image by $\Upsilon$ of the product measure
$\mu_{\mathscr{L}}\otimes Leb_{\R^2}$ on $\mathscr{L}\times\R^2$ is equivalent to the measure $\nu_{\mathcal{M}}$.

The local product structure on $\mathcal{M}$ arising from $\Upsilon$ helps to deduce Theorem~\ref{thm:existxi} form Lemma~\ref{lem:existbound}. Indeed,
suppose, contrary to the claim of Theorem~\ref{thm:existxi}, that there exists a measurable subset $\mathscr{L}_0\subset \mathscr{L}$ of positive
$\mu_{\mathscr{L}}$-measure such that for every $\Lambda\in\mathscr{L}_0$ condition \eqref{asstion-thm} does not hold for the surface $M(\Lambda,R)$.
In view of \eqref{eq:changesigma}, condition \eqref{asstion-thm} does not hold for every surface $(h_sg_t)\cdot M(\Lambda,R)$ with $s,t\in\R$.
Consequently, for every surface from the set $\Upsilon(\mathscr{L}_0\times[0,1]^2)$ condition \eqref{asstion-thm} is not valid. Since the set has positive $\nu_{\mathcal{M}}$-measure, one gets a contradiction to Lemma~\ref{lem:existbound}.
\end{proof}

\begin{remark}\label{rem:saddle}
Note that for every translation surface $(M,\omega)$ for a.e.\ $\theta\in S^1$ the vertical flow on $(M,r_\theta\omega)$ has no saddle connection. It follows from Fubini theorem, that if $\nu$ is an $SL(2,\R)$-invariant measure on the moduli space $\mathcal{M}_1(M)$ then for $\nu$-a.e.\ $(M,\omega)$ the
vertical flow on $(M,\omega)$ has no saddle connection. Applying this observation to the measure $\nu_{\mathcal{M}}$ and then proceeding as in the proof of
Theorem~\ref{thm:existxi}, we obtain that for every $R>0$ and $\mu_{\mathscr{L}}$-a.e.\ lattice $\Lambda\in\mathscr{L}$ the surface $M(\Lambda,R)$ has no
saddle connections.
\end{remark}

\section{Periodic case  and examples}\label{sec:examples}
In this section we will show how to determine the direction of confining strips when
the surface $M(\Lambda,R)$ is a periodic point of the Teichm\"uller flow. Under additional
assumption (hyperbolicity of a certain matrix) we present a procedure that helps to construct
examples of pairs $\Lambda$, $R$ for which the direction of confining strips on $L(\Lambda,R)$
can be computed. We conclude this section with a specific example.
\smallskip

Suppose that $M(\Lambda,R)=(M,\omega_0)\in\mathcal{M}$ is a periodic
point of the Teichm\"uller flow with period $t_0>0$. Since
$g_{t_0}(M,\omega_0)=(M,\omega_0)$ in the moduli space, there exists an affine homeomorphism $\psi:(M,\omega_0)\to (M,\omega_0)$ such that $D\psi=g_{t_0}$.
Note that the group of affine homeomorphisms with derivative being the identity
consists of $id$ and $\tau$. Since $\psi\circ \tau$ is affine with $D(\psi\circ
\tau)=g_{t_0}$,  we have that either $\psi\circ \tau=\tau\circ\psi$ or $\psi\circ \tau=\psi$. As $\psi$ is one-to-one, the latter equality is not
satisfied, so $\psi\circ \tau=\tau\circ\psi$. Therefore $\psi_*\circ \tau_*=\tau_*\circ\psi_*$, hence $\psi_*$ preserves the subspaces $V$ and $V^\perp$.
Since $\psi_*$ restricted to $V\cap H_1(M,\Z)$ is an automorphism, any  matrix representation of $\psi_*:V\to V$ is an element of $SL(2,\Z)$. Suppose that
this element is hyperbolic. Then there exist eigenvectors $\xi_s,\xi_u\in V$  such that
\[\psi_*\xi_s=\lambda\xi_s\quad\text{ and }\quad\psi_*\xi_u=\lambda^{-1}\xi_u\quad\text{ with }|\lambda|<1.\]

\begin{lemma}\label{lem:pseudo}
There exists $C>0$ such that $|\langle\sigma_T^{\omega_0}(x),\xi_s\rangle|\leq C$  for all $x\in M^+_{\omega_0}$ and $T>0$.
\end{lemma}

\begin{proof}
Denote by $\mathcal{M}_0$ the Teichm\"uller flow orbit of $(M,\omega_0)$.  This is an invariant set with the unique probability invariant measure
$\nu_{\mathcal{M}_0}$ which is the image  of the Lebesgue measure on the circle $\R/t_0\Z$ via the map $\R/t_0\Z\ni t\mapsto
g_t(M,\omega_0)\in\mathcal{M}$. Let us consider the KZ-cocycle $(G_t^{\mathcal{V}})_{t\in\R}$ on $\mathcal{V}$ over the set $\mathcal{M}_0$. The cocycle is
completely determined by the map $\psi_*:V\to V$. Since it is hyperbolic we have
\[E^-_\omega=\R\xi_s\quad\text{ and }\quad E^+_\omega=\R\xi_u\quad\text{ for every }\quad\omega\in \mathcal{M}_0.\]
In view of Theorem~\ref{thm:existbound}, for almost every  $t\in\R/t_0\Z$ there exists  $C>0$ such that
$|\langle\sigma_T^{g_t\omega_0}(x),\xi_s\rangle|\leq C$ for all $x\in M$ and $T\in\R$. By \eqref{eq:changesigma}, this yields the assertion of the lemma.
\end{proof}

\medskip

Fix $R>0$ and consider a lattice $\Lambda\in\mathscr{L}$ satisfying \eqref{cond:adm}. Let us choose a positive basis $\gamma_+,\gamma_-$ of $\Lambda$
satisfying \eqref{eq:contint}. Then there exists a unique $g\in SL(2,\R)$ such that $g(\gamma_+)=(1,0)$ and $g(\gamma_-)=(0,1)$. Denote by $\vartheta\in
S^1$ the direction of the vector $g(0,1)$ and $u:=g(R,0)$. Then $u\in(-1/2,1/2)^2\setminus\{(0,0)\}$ and $u\wedge\vartheta>0$.

Let us consider the surface $g\cdot M(\Lambda,R)\in\mathcal{M}$.
This translation surface is the union of two copies of the
standard torus $\R^2/\Z^2$ with translation structure inherited
from the Euclidean plane (as a model torus we will use the square
$[-1/2,1/2)^2$) glued along the linear slit between $u$ and $-u$,
see Figure~\ref{surfend}. Such surface will be denoted by $M(u)$.
\begin{figure}[h]
\includegraphics[width=0.5\textwidth]{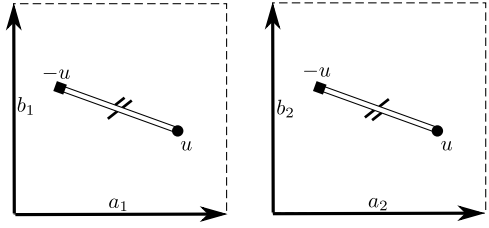}
\caption{Translation surface $M(u)$\label{surfend}}
\end{figure}

Note that the map
\begin{equation}\label{not:map}
(\Lambda,\gamma_+,\gamma_-)\mapsto (u,\vartheta)
\end{equation}
gives a one-to-one correspondence between translation surfaces
$M(\Lambda,R)$ with assigned oriented positive basis $\gamma_+,\gamma_-$
and translation surfaces $M(u)$ with assigned direction
$\vartheta$ such that $u\wedge\vartheta>0$. The
inverse of \eqref{not:map} is given by $\eta_{u,\vartheta}\in
SL(2,\R)$ such that $\eta_{u,\vartheta}u=(R,0)$ and
$\eta_{u,\vartheta}\vartheta=(0,(u\wedge\vartheta)/R)$. Then taking
$\gamma_+=\eta_{u,\vartheta}(1,0)$,
$\gamma_-=\eta_{u,\vartheta}(0,1)$ and
$\Lambda_{u,\vartheta}=\Z\gamma_++\Z\gamma_-$, we get
$M(\Lambda_{u,\vartheta},R)=\eta_{u,\vartheta}M(u)$.\\

By $\mathscr{U}\subset\mathcal{M}$ we denote the set of translation
surfaces $M(u)$ for which
\[u\in\T^2_0:=[-1/2,1/2)\times[-1/2,1/2)\setminus\{(0,0),(-1/2,-1/2),(-1/2,0),(0,-1/2)\}.\]
The set $\mathscr{U}$  is $SL(2,\Z)$-invariant
and  $g\cdot M(u)=M(g u)$, where $g u\in \T^2_0$ is the image of $u$ by
the algebraic automorphism $g:\T^2\to\T^2$ defined by the matrix
$g\in SL(2,\Z)$, see \cite{Fr-Ul:erg_bil}.

The $SL(2,\Z)$-action on $\mathscr{U}$ defines the induced
homology action on $H_1(M,\R)$ for which the subspace $V$ is
invariant, see also \cite{Fr-Ul:erg_bil}. For every translation
surface $M(u)$ we choose canonically the basis $\{e_1,e_2\}$ of
$V$, where $e_1:=a_1-a_2$ and $e_2=b_1-b_2$. Since
$\gamma_+=\eta_{u,\vartheta}(1,0)$ and
$\gamma_-=\eta_{u,\vartheta}(0,1)$, we have
$\gamma_1=(\eta_{u,\vartheta})_*e_1$ and
$\gamma_2=(\eta_{u,\vartheta})_*e_2$.

For every $g\in SL(2,\Z)$ and $u\in \T^2_0$ denote by $g_*(u)$ the
matrix representation in canonical bases of the homology map
induced by $g:M(u)\to M(gu)$. Since $g_*(u)$ is well defined up to
multiplication by $\pm 1$ (see \cite{Fr-Ul:erg_bil}), it is
treated as an element of $PSL(2,\Z)$.

\begin{theorem}\label{thm:eigen}
Let $u\in (-1/2,1/2)^2\setminus\{(0,0)\}$ be rational. Suppose that $h\in SL(2,\Z)$ is a hyperbolic matrix such that $hu=u$ (in $\T^2_0$) and $h_*(u)$ is also
hyperbolic. Denote by $\vartheta,\theta\in S^1$ the contracting eigendirections of $h$ and $h_*(u)$ respectively. Then all vertical light rays in
$F(\Lambda_{u,\vartheta},R)$ are trapped in  infinite bands in direction $\eta_{u,\vartheta}\theta$.
\end{theorem}
\begin{proof}
Denote by $t_0>0$ the natural logarithm of the largest eigenvalue
of $h$. Since $hu=u$ in $\T^2_0$, there exists an affine automorphism
$\phi:M(u)\to M(u)$ such that $D\phi=h$. Let us consider the
affine automorphism $\eta_{u,\vartheta}\circ\phi\circ
\eta_{u,\vartheta}^{-1}:M(\Lambda,R)\to M(\Lambda,R)$ with
$\Lambda:=\Lambda_{u,\vartheta}$. Then
$D\big(\eta_{u,\vartheta}\circ\phi\circ
\eta_{u,\vartheta}^{-1}\big)=\eta_{u,\vartheta}\circ h\circ
\eta_{u,\vartheta}^{-1}$ and $(0,1)$ is its contracting eigenvector
for the eigenvalue $e^{-t_0}$. Take an expanding eigenvector of the
form $(1,r)$ for some $r\in\R$. Let us consider the surface
$(M,\omega_0)=h^T_{-r}M(\Lambda,R)$ ($h^T_{-r}$ is the transpose of $h_{-r}$) and its affine automorphism
\[\psi:=h^T_{-r}\circ \eta_{u,\vartheta}\circ\phi\circ
\eta_{u,\vartheta}^{-1}\circ h^T_r.\]
Then  $(0,1)$  is a contracting eigenvector of $D\psi$ and $(1,0)$ is an expanding one. Therefore, $D\psi=g_{t_0}$. It follows that $(M,\omega_0)$ is a periodic
element for the Teichm\"uller flow with period $t_0$. Since
$\psi$ is conjugate to $\phi$ by $h^T_{-r}\circ
\eta_{u,\vartheta}$ and the subspace $V\subset H_1(M,\R)$ is
$SL(2,\R)$-invariant, the induced automorphisms $\phi_*:V\to V$
and $\psi_*:V\to V$ are isomorphic via the map $(h^T_{-r}\circ
\eta_{u,\vartheta})_*:V\to V$. As $h_*(u)$ is the matrix
representation of $\phi_*:V\to V$ in the basis $e_1,e_2$ and $\theta=(\theta_1,\theta_2)$ is a contracting eigenvector for $h_*(u)$,
the homology class $\theta_1e_1+\theta_2e_2\in V$ is a contracting eigenvector for $\phi_*:V\to V$.
It follows that $\xi_s:=(h^T_{-r}\circ
\eta_{u,\vartheta})_*(\theta_1e_1+\theta_2e_2)\in V$ is a contracting eigenvector for
$\psi_*:V\to V$.
In view of Lemma~\ref{lem:pseudo}, there exists $C>0$ such that
\[|\langle\sigma_T^{\omega_0}(x),\xi_s\rangle|\leq C\text{  for all }x\in M^+_{\omega_0}\text{ and }T>0.\]
Let us consider the map $h^T_r:(M,\omega_0)\to M(\Lambda,R)$ and its induced transformation $(h^T_r)_*$ on homologies.
Since $(h^T_r)_*\sigma_T^{\omega_0}(x)=\sigma_T^{h^T_r\omega_0}(h^T_rx)$ and $(h^T_r)_*$ preserves
algebraic intersection number, we
also have \[|\langle\sigma_T^{h^T_r\omega_0}(x),(h^T_r)_*\xi_s\rangle|\leq C\text{
for all }x\in M^+_{h^T_r\omega_0}\text{ and }T>0\]
with $(M,h^T_r\omega_0)=M(\Lambda,R)$.

Since $(\eta_{u,\vartheta})_*$ sends the basis
$\{e_1,e_2\}$ to $\{\gamma_1,\gamma_2\}$, we have
$(h^T_r)_*\xi_s=\theta_1\gamma_1+\theta_2\gamma_2$. In view of
Theorem~\ref{thm:existband}, it follows that there exists
$\overline{C}>0$ such that every vertical orbit on $F(\Lambda,R)$
is trapped in an infinite band in the direction of the vector
\[\bar{v}(\Lambda,(h^T_r)_*\xi_s)=\langle\gamma_2,(h^T_r)_*\xi_s\rangle\gamma_+-\langle\gamma_1,(h^T_r)_*\xi_s\rangle\gamma_-=
-2(\theta_1\gamma_++\theta_2\gamma_-)=-2\eta_{u,\vartheta}\theta\]
and whose width is bounded by $\overline{C}$.
\end{proof}

Recall hat the matrices
$h^+:={\scriptsize\begin{pmatrix}1&1\\0&1\end{pmatrix}}$,
$h^-:={\scriptsize\begin{pmatrix}1&0\\1&1\end{pmatrix}}$ generate
$SL(2,\Z)$. Therefore, the following result allows us to calculate
$g_*(u)$ for every $g\in SL(2,\Z)$ and $u\in\T^2_0$. Then using
Theorem~\ref{thm:eigen}, one can construct an explicit example of
a lattice $\Lambda$ and $R>0$ such that vertical light rays in
$L(\Lambda,R)$ are trapped in strips, see Example~\ref{ex}.
\begin{proposition}[see \cite{Fr-Ul:erg_bil}]\label{lem:basic}
Set $S:=\{(x,y)\in \T^2_0:-1/2\leq x+y< 1/2\}$. For every
$u\in\T^2_0$ we have
\[h^{\pm}_*(u)=\left\{\begin{matrix}h^\pm & \text{ if } & u\in S,\\
(h^\pm)^{-1} & \text{ if } & u\notin S.
\end{matrix}\right.\]
\end{proposition}

\begin{example}\label{ex}
Let $R=1/3$ and let us start from the point $u=(1/3,0)\in \T^2_0$.
Take
$h=(h^-)^3h^+={\scriptsize\begin{pmatrix}1&1\\3&4\end{pmatrix}}$.
Then $hu=u$, and using Proposition~\ref{lem:basic}, one can
compute that
$h_*(u)=h^-h^+={\scriptsize\begin{pmatrix}1&1\\1&2\end{pmatrix}}$.
Therefore, $\vartheta$ and $\theta$ are directions of contracting
eigenvectors $ \big(-\frac{3+\sqrt{21}}{6},1\big)$  and
$\big(-\frac{1+\sqrt{5}}{2},1\big)$ resp. Then
$\eta_{u,\vartheta}={\scriptsize\begin{pmatrix}1&(3+\sqrt{21})/6\\0&1\end{pmatrix}}$
and
\[\eta_{u,\vartheta}\Big(-\frac{1+\sqrt{5}}{2},1\Big)=\Big(\frac{\sqrt{21}-3\sqrt{5}}{6},1\Big).\]
It follows that for the lattice
\[\Lambda=\eta_{u,\vartheta}\Z^2=(1,0)\,\Z+\Big(\frac{3+\sqrt{21}}{6},1\Big)\,\Z\]
every vertical light ray in $L(\Lambda,1/3)$ is trapped in a band
with slope  $-\frac{\sqrt{21}+3\sqrt{5}}{4}$.
\end{example}
\subsection*{Acknowledgements}
The authors thank Jens Marklof for noticing the connection between
refractive properties of Eaton lens distributions and singular
complex planes (as studied in \cite{J-S}) attending a presentation
of the second author in Bristol in December 2012. His
observation led to this article. The authors also thank Corinna
Ulcigrai for useful discussions and the referees for valuable suggestions to
improve the presentation.

\end{document}